\documentclass{amsart}
\usepackage{graphicx} 
\usepackage{amssymb,amsthm,amsmath,mathtools,float,comment,soul,xcolor,verbatim,graphicx}
\usepackage{hyperref}
\usepackage{cleveref}
\title{D-tensor paraproducts and its caricatures}
\author{Oluwadamilola Fasina}
\date{January 2026}

\usepackage[left=3cm, right=3cm, top=2cm, bottom=2cm]{geometry}
\newtheorem{theorem}{Theorem}[section]

\newtheorem{lemma}[theorem]{Lemma}
\newtheorem{corollary}[theorem]{Corollary}
\newtheorem{definition}[theorem]{Definition}

\begin{document}

\maketitle

\begin{abstract}
We generalize the $2$-tensor paraproduct decomposition result of  \cite{fasina2025quasilinearization} to $d$-tensors. In particular, we show that for $A \in C^{d}(\mathbb{R}), f \in \Lambda_{\alpha}([0,1]^d)$, $A(f)$ can be approximated by   $\tilde{A}_{(N_i)_{i=0}^d}(f) = (\sum_{\beta=1}^d A^{\beta}(P^{j_1,j_2, \ldots, j_d}(f)) \tilde{\mathbf{v}}^{\beta}(f) ) $ with the residual $\Delta_{(N_i)_{i=1}^d}(A,f) = \tilde{A}_{(N_i)_{i=1}^d}(f) - A(f) \in \Lambda_{2\alpha}([0,1]^d)$. Our theoretical findings are supported by a computational example for d=3.

\end{abstract}

\section{Introduction}

We generalize the 2-tensor paraproduct decomposition developed in \cite{fasina2025quasilinearization}, which built on Bony's seminal work, \cite{bony1981calcul}, to $d$-tensors for the situation $(f \in \Lambda_{\alpha}([0,1]^d), A \in C^d(\mathbb{R}))$; in particular, we prove the following theorem:

\begin{theorem}
Suppose  $A \in C^d(\mathbb{R})$, $ f \in  \Lambda_\alpha([0,1]^d), 0 < \alpha < \frac{1}{2}$, then for the operator $T: f \to A(f)$ we can approximate $A(f)$ with

\begin{align}
&  \tilde{A}_{(N_i)_{i=0}^d}(f) =  \sum_{j_1, \ldots, j_d=0}^{N_1,\ldots,N_d} \sum_{\beta=1}^d A^{\beta}(P^{j_1,j_2, \ldots, j_d}(f)) \tilde{\mathbf{v}}^{\beta}(f) \in \Lambda_{\alpha}([0,1]^d)
\label{eq:2}
\end{align}

such that the multiscale tensor paraproduct transforms $T : f \to A(f)$ to 

\begin{align}
\Pi^{(N_i)_{i=1}^d}_{(A^\beta)_{\beta=1}^d} : f \to \tilde{A}_{(N_i)_{i=0}^d}(f) + \Delta_{(N_i)_{i=0}^d}(A,f)
\label{eq:3}
\end{align}

where $\Delta_{(N_i)_{i=0}^d}(A,f) = A(f) - \tilde{A}_{(N_i)_{i=0}^d}(f) \in \Lambda_{2\alpha}([0,1]^2)$ is the residual containing twice the regularity of $f$.

\end{theorem}

 Furthermore, examinations of the organs of $\Pi^{(N_i)_{i=1}^d}_{(A^\beta)_{\beta=1}^d} $ leads to caricatures highlighted herein - namely, that (a) $\tilde{A}_{(N_i)_{i=0}^d}(f)$ is a taylor expansion of $A$ in terms of the derivatives of $A$ and scaling parameters of the wavelet operators and (b) the 3-tensor paraproduct decomposition.

\section{Acknowledgments}

The author wishes to thank Ronald R. Coifman for useful discussions.

\section{Preliminaries}

Included in this section are the relevant definitions needed to build $\Pi^{(N_i)_{i=1}^d}_{(A^\beta)_{\beta=1}^d}$:

\begin{definition}
Dyadic distance along direction $i$:

\begin{align}
d_d(x'_i, x_i) = \inf_{j_i,k_i} 
\begin{cases} 
|I^{j_i}_{k_i}| &  x'_i , x_i \in I^{j_i}_{k_i}   \\
0, & x'_i = x_i
\end{cases}
\label{eq:7}
\end{align}

where $I^{j_i}_{k_i} = (2^{j_i}k_i, 2^{j_i}(k_i+1)]$ is the dyadic interval at scale $j_i$ at location $k_i$.

\end{definition}

\begin{definition}
The 1D scaling and wavelet functions have the usual definitions:

\begin{align}
\phi(x) \coloneq 
\begin{cases}
 1 & \text{if } x \in (0,1] \\
 0 & \text{if } x \notin (0,1]
\end{cases}
\qquad
\psi(x) \coloneq
\begin{cases}
 1 & \text{if } 0 < x \leq \frac{1}{2} \\
 -1 & \text{if }  \frac{1}{2} < x\leq 1
\end{cases}
\end{align}

\begin{align}
\phi^j_k(x) \coloneq 2^{\frac{-j}{2}} \phi(2^{-j}x - k), j,k \in \mathbb{N} 
\qquad
\psi^j_k(x) \coloneq 2^{\frac{-j}{2}} \psi(2^{-j}x - k), j,k \in \mathbb{N} 
\end{align}
\end{definition}

\begin{definition} 
We define the $d$-dimensional scaling and wavelet functions as the cartesian product of $d$ one-dimensional scaling functions and wavelet functions, respectively.

\begin{align}
\phi^j_k(x_1,x_2, \ldots, x_d) \coloneq \phi^j_k(x_1) \times \phi^j_k(x_2) \times \ldots \times \phi^j_k(x_d)
\end{align}

\begin{align}
\psi^j_k(x_1,x_2, \ldots, x_d) \coloneq \psi^j_k(x_1) \times \psi^j_k(x_2) \times \ldots \times \psi^j_k(x_d)
\end{align}

\end{definition}

\begin{definition} Let $\delta_{x_1} f(x_1,x_2, \ldots, x_d) = f(x'_1,x_2, \ldots, x_d) - f(x_1,x_2, \ldots, x_d)$ such that $\delta_{x_1}$ is the difference operator in direction $x_1$ We say a function is mixed $\alpha$-H\"older if the following conditions are satisfied:

\begin{align}
\delta_{x_d} \delta_{x_{d-1}} \ldots \delta_{x_1} f(x_1,x_2, \ldots, x_d) \leq C (d_d(x'_1,x_1)^{\alpha}d_d(x'_2,x_2)^{\alpha} \ldots d_d(x'_d,x_d)^{\alpha}), 0 < \alpha < \frac{1}{2}
\end{align}

\begin{align}
\delta_{x_i} f(x_1,x_2, \ldots, x_d) \leq C d_d(x'_i,x_i)^{\alpha} \hspace{0.2 cm} , \hspace{0.2 cm} i=1,\ldots, d
\end{align}

for $C, \alpha > 0$. We denote the space of functions satisfying this condition as $\Lambda_{\alpha}([0,1]^d)$

\end{definition}

We now define the averaging operator acting on a function supported on a $d$-dimensional box. Let $I_1 \times I_2 \times \ldots \times I_d = [0,1]^d$. Along each direction, $i$, we construct $j_i=0,1,2, \ldots, N$ scales such that at the $j_i$th scale one has $k_i=1,2, \ldots, 2^{j}$ intervals. Let the dyadic interval along direction $i$ be denoted by: $I^{j_i}_{k_i} = (2^{j_i}k_i, 2^{j_i}(k_i+1)]$ and let $\chi^{j_i}_{k_i}(x_i)$ as the characteristic function associated with interval $I^{j_i}_{k_i}$.

\begin{definition}
Let the averaging operator, $P^{j_1,j_2, \ldots, j_d}$, be understood with the relation:

\begin{align}
&P^{j_1,j_2, \ldots, j_d}(f) \coloneq \sum_{k_1 = 1}^{2^{j_1}} \sum_{k_2 = 1}^{2^{j_2}} \ldots \sum_{k_d = 1}^{2^{j_d}}   P^{j_1,j_2, \ldots, j_d}_{k_1,k_2, \ldots, k_d}(f) \nonumber \\
& = \sum_{k_1 = 1}^{2^{j_1}} \sum_{k_2 = 1}^{2^{j_2}} \ldots \sum_{k_d = 1}^{2^{j_d}}  \frac{1}{|I^{j_1}_{k_1} | \times |I^{j_2}_{k_2} | \times \ldots \times |I^{j_d}_{k_d} |}( \int_{|I^{j_1}_{k_1} |} \int_{|I^{j_2}_{k_2} |} \ldots \int_{|I^{j_d}_{k_d} |} \nonumber \\
& \phi(y_1, y_2, \ldots, y_d)f(y_1, y_2, \ldots, y_d) dy_1 dy_2, \ldots dy_d) \chi^{j_1}_{k_1}(x_1) \times \chi^{j_2}_{k_2}(x_2) \times \ldots \times \chi^{j_d}_{k_d}(x_d)
\end{align}
\label{defn:avg_op}
\end{definition}

\section{d-tensor paraproducts}

Included in this section is a statement and proof of our main result:

\begin{theorem}
Suppose  $A \in C^d(\mathbb{R})$, $ f \in  \Lambda_\alpha([0,1]^d), 0 < \alpha < \frac{1}{2}$, then for the operator $T: f \to A(f)$ we can approximate $A(f)$ with

\begin{align}
&  \tilde{A}_{(N_i)_{i=0}^d}(f) =  \sum_{j_1, \ldots, j_d=0}^{N_1,\ldots,N_d} \sum_{\beta=1}^d A^{\beta}(P^{j_1,j_2, \ldots, j_d}(f)) \tilde{\mathbf{v}}^{\beta}(f) \in \Lambda_{\alpha}([0,1]^d)
\label{eq:2}
\end{align}

such that the multiscale tensor paraproduct transforms $T : f \to A(f)$ to 

\begin{align}
\Pi^{(N_i)_{i=1}^d}_{(A^\beta)_{\beta=1}^d} : f \to \tilde{A}_{(N_i)_{i=0}^d}(f) + \Delta_{(N_i)_{i=0}^d}(A,f)
\label{eq:3}
\end{align}

where $\Delta_{(N_i)_{i=0}^d}(A,f) = A(f) - \tilde{A}_{(N_i)_{i=0}^d}(f) \in \Lambda_{2\alpha}([0,1]^2)$ is the residual containing twice the regularity of $f$.

\label{thm:4.1}
\end{theorem}

\begin{proof}

We first prove a short lemma which establishes $A(f)$ can be expanded into a telescoping series with respect the mixed difference of the scales of the averaging operator. 

\begin{lemma} Let $A \in C^{d}(\mathbb{R})$ and $f \in \Lambda_{\alpha}([0,1]^d)$. Define $Q^{j_i}$ as the finite difference operator between scales of the averaging operator:

\begin{align}
& Q^{j_i}(f) \coloneq \sum_{ \tilde{k}_i=1_i}^{2^{j_i}} P_{k_i}^{{j_i} + 1}(f) - \sum_{k_i=1_i}^{2^{j_i} + 1} P_{k_i}^{j_i}(f) \nonumber \\
& =  \sum_{\tilde{k}_i=1_i}^{2^{j_i}}  \frac{1}{|I^{ {j_i} + 1}_{k_i}| } (\int_{|I^{{j_i} + 1}_{k_i}|} \phi^{{j_i} +  1}_{k_i}(x_1, x_2, \ldots, y_i, \ldots, x_d)  f(x_1, x_2, \ldots, y_i, \ldots, x_d)  dy_i ) \times \chi^{j_i + 1}_{k_i}(x_i) \nonumber \\
&- \sum_{k_i=1_i}^{2^{j_i} }  \frac{1}{|I^{ {j_i} }_{k_i}| }( \int_{|I^{{j_i} }_{k_i}|} \phi^{j_i}_{k_i}(x_1, x_2, \ldots, y_i, \ldots, x_d) f(x_1, x_2, \ldots, y_i, \ldots, x_d)  dy_i ) \times \chi^{j_i}_{k_i}(x_i)  \nonumber \\
& =   \sum_{k_i=1_i}^{2^{j_i} }  \frac{1}{|I^{ {j_i} }_{k_i}| }\int_{|I^{ {j_i} }_{k_i}| } \psi^{j_i}_{k_i}(x_1, x_2, \ldots, y_i, \ldots, x_d) f(x_1, x_2, \ldots, y_i, \ldots, x_d)  dy_i \times \chi^{{j_i} }_{k_i}(x_i) 
\label{eq:18}
\end{align}

where the kernel of $Q^{j_i}$ is the wavelet function in direction $i$. Then the following identity holds:
\begin{align}
& A(f) = \sum_{j_1, \ldots, j_d=0}^{N_1, \ldots, N_d} A(Q^{j_d}Q^{j_{d-1}} \ldots Q^{j_1})(f)
\end{align}

if we assume the average of $f$ over the entire interval in any direction is 0. 

\label{lemma:4}
\end{lemma}

\begin{proof}
Let $Q^{j_m}(f)$ be the finite difference operator between scales of the averaging operator in direction $m$, where $P^{j_m}(f)$ denotes an average over $f$ in direction $m$.

\begin{align}
& Q^{j_i}(f) = P^{j_i + 1}(f) - P^{j_i}(f) \nonumber \\
& =  \sum_{k_i=0_i}^{2^{N_i} }  \frac{1}{|I^{ {j_i} }_{k_i}| }\int_{|I^{ {j_i} }_{k_i}| } \psi^{j_i}_{k_i}(x_1, x_2, \ldots, y_i, \ldots, x_d) f(x_1, x_2, \ldots, y_i, \ldots, x_d)  dy_i \times \chi^{{j_i} }_{k_i}(x_i) 
\end{align}

Let $A({Q^{j_i}}(f)) = A(P^{j_i + 1}(f)) - A(P^{j_i}(f))$ denote the pointwise application of $A$ on the finite difference operator in direction $i$. The result is established by inducting on the directions the finite difference operator is applied to $f$.

\begin{align}
& \sum_{j_m = 0}^{N_i} A({Q^{j_i}})(f) = \sum_{j_i = 0}^{N_i} A(P^{j_i + 1}(f)) - A(P^{j_i}(f))  \nonumber \\
& = A(P^{N_i + 2}(f)) - A(P^{0_i}(f)) \nonumber   \\
& = A(f)
\end{align}

where the relation in the second line holds by expansion of the telescoping series, and the relation in the third line hold since $A(P^{j_i}(f)) \to A(f)$ as $j_i \to \infty$ and $A(P^{0_i}(f)) = 0 $ by assumption. Now assume this holds for $k$ applications of finite difference operators:

\begin{align}
 \sum_{j_k = 0}^{N_k} \sum_{j_{k-1} = 0}^{N_{k-1}}  \ldots \sum_{j_0 = 0}^{N_0}  A({Q^{j_k}Q^{j_{k-1}} \ldots Q^{j_1}})(f) = A(f)
\end{align}
We show the relation holds for $k+1$ applications:

\begin{align}
 & \sum_{j_{k+1} = 0}^{N_{k+1}} (\sum_{j_k = 0}^{N_k} \sum_{j_{k-1} = 0}^{N_{k-1}}  \ldots \sum_{j_0 = 0}^{N_0}A(Q^{j_{k+1}}Q^{j_k}Q^{j_{k-1}} \ldots Q^{j_1})(f)  \nonumber \\
 & = \sum_{j_{k+1} = 0}^{N_{k+1}}A(Q^{j_{k+1}} ((\sum_{j_k = 0}^{N_k} \sum_{j_{k-1} = 0}^{N_{k-1}}  \ldots \sum_{j_0 = 0}^{N_0}Q^{j_k}Q^{j_{k-1}} \ldots Q^{j_1})(f)) \nonumber \\
 & = \sum_{j_{k+1} = 0}^{N_{k+1}}A(Q^{j_{k+1}}) (f) \nonumber \\
 & = \sum_{j_{k+1}= 0}^{N_{k+1}} A(P^{j_{k+1} + 1}(f)) - A(P^{j_{k+1}}(f))  \nonumber \\
 & = A(P^{N_{k+1} + 2}(f)) - A(P^{0_{k+1}}(f)) \nonumber   \\
 & = A(f)
\end{align}
Since we've shown that if the relation holds for the pointwise application of $A$ to $k$ sequential applications of the mixed difference operator to $f$ it holds for the pointwise application of $k+1$ sequential applications, the result holds for the pointwise application of an arbitrary number, $d$, sequential applications of the mixed difference operator to $f$.
\end{proof}

The fundamental theorem of calculus can be used to connect the identity established in lemma  \ref{lemma:4} to a quasilinearization of $A(f)$, but to do so we need to define a $d$-dimensional interpolation operator. This leads us to the next lemma.

\begin{lemma}
Define $\delta^{t_i}(f) = P^{j_i}(f) + t_i(P^{j_i + 1}(f) - P^{j_i}(f))$ to be an interpolation operator acting at scale $j_i$.  Then the following relation holds:

\begin{align}
\delta^{t_d} \delta^{t_{d-1}} \ldots \delta^{t_1}(f) \left. \right|_{t_d=0}^{t_d=1}  \left. \right|_{t_{d-1}=0}^{t_{d-1}=1} \ldots  \left. \right|_{t_1=0}^{t_1=1} = Q^{j_d}Q^{j_{d-1}} \ldots Q^{j_1}(f) 
\end{align}
\label{lemma:5}
\end{lemma}

\begin{proof}

This can be proved by inducting on the sequential applications of the interpolation operators. First we show the base case holds:

\begin{align}
&\delta^{t_1}(f) \left. \right|_{t_1=0}^{t_1=1} \coloneq P^{j_1}(f) + t_1(P^{j_1 + 1}(f) - P^{j_1}(f))\left. \right|_{t_1=0}^{t_1=1} \nonumber \\
& = (P^{j_1 + 1}(f) - P^{j_1}(f) \nonumber \\
& = Q^{j_1}(f)
\end{align}

Assume the result holds for $k$ sequential applications of interpolation operators:

\begin{align}
\delta^{t_k} \delta^{t_{k-1}} \ldots \delta^{t_1} (f) \left. \right|_{t_1 = 0}^{t_1 = 1} \left. \right|_{t_{2} = 0}^{t_{2} = 1} \ldots \left. \right|_{t_k = 0}^{t_k = 1} = Q^{j_k}Q^{j_{k-1}} \ldots Q^{j_1}(f)
\end{align}

We can then show the relation holds for $k+1$ sequential applications of interpolation operators:

\begin{align}
& \delta^{t_{k+1}} (\delta^{t_k} \delta^{t_{k-1}} \ldots \delta^{t_0} (f) \left. \right|_{t_1 = 0}^{t_1 = 1} \left. \right|_{t_{2} = 0}^{t_{2} = 1} \ldots \left. \right|_{t_k = 0}^{t_k = 1}) \left. \right|_{t_{k+1}=0}^{t_{k+1} =1} = \delta^{t_{k+1}}(Q^{j_k}Q^{j_{k-1}} \ldots Q^{j_1}(f)) \left. \right|_{t_{k+1}=0}^{t_{k+1}=1} \nonumber  \\
& = P^{j_{k+1}}(Q^{j_k}Q^{j_{k-1}} \ldots Q^{j_1}(f)) + t_{k+1}(P^{j_{k+1}+1}(Q^{j_k}Q^{j_{k-1}} \ldots Q^{j_1}(f)) - P^{j_{k+1}}(Q^{j_k}Q^{j_{k-1}} \ldots Q^{j_1}(f)) ) \left. \right|_{t_{k+1}=0}^{t_{k+1}=1} \nonumber \\
& =  (P^{j_{k+1}+1}(Q^{j_k}Q^{j_{k-1}} \ldots Q^{j_1}(f)) - P^{j_{k+1}}(Q^{j_k}Q^{j_{k-1}} \ldots Q^{j_1}(f)) ) \nonumber \\
& = (Q^{j_{k+1}}Q^{j_{k}} \ldots Q^{j_1}(f))
\end{align}

Since the relation holds for $k+1$ sequential applications of the interpolation operator given the assumption it holds for $k$ sequential applications of the interpolation operator, by induction it holds for $d$ sequential applications of the interpolation operator. Thus, we have:

\begin{align}
\delta^{t_d} \delta^{t_{d-1}} \ldots \delta^{t_1} (f) \left. \right|_{t_1 = 0}^{t_1 = 1} \left. \right|_{t_{2} = 0}^{t_{2} = 1} \ldots \left. \right|_{t_d = 0}^{t_d = 1} = Q^{j_d}Q^{j_{d-1}} \ldots Q^{j_1}(f)
\end{align}

which completes the proof.

\end{proof}

\begin{definition}

Let $\delta^t(f)$ be an interpolation function which interpolates between the scales $(j_1,j_2, \ldots, j_d)$ of the averaging operators with scaling parameters $(t_1,t_2, \ldots, t_d)$ with $0 \leq t_k \leq 1$ for $k=1,2 \ldots, m$:

\begin{align}
\delta^t(f) \coloneq \delta^{t_d} \delta^{t_{d-1}} \ldots \delta^{t_1} (f)
\end{align}

such that lemma \ref{lemma:5} holds

\end{definition}

 From the previous two lemmas and the fundamental theorem of calculus $A(f)$ has an identity which is a summation (with respect to the scaling parameters in all directions) over an integration of the power series comprised of the operators $(\mathbf{v}^1, \mathbf{v}^2, \ldots, \mathbf{v}^d$ with respect to the interpolation parameters in all directions. First, we define the operators $(\mathbf{v}^1, \mathbf{v}^2, \ldots, \mathbf{v}^d)$ which come from differentiating $A(\delta^t(f))$. The elements in $(\mathbf{v}^1,\mathbf{v}^2, \ldots, \mathbf{v}^d)$ are defined such that each power of $\mathbf{v}$ corresponds to the number of independent partial derivatives of $\delta^t(f)$ taken:

\begin{align}
\mathbf{v}^1(f) \coloneq \frac{\partial^d}{\partial t_d \partial t_{d-1} \ldots \partial t_1}(\delta^t(f)) 
\end{align}

\begin{align}
& \mathbf{v}^2(f) \coloneq \frac{\partial^{d-2}}{\partial t_d \partial t_{d-1} \ldots \partial t_3}(\frac{\partial}{\partial t_2}(\delta^t(f))\frac{\partial}{\partial t_1}(\delta^t(f))) + \nonumber \\
 &\frac{\partial^{d-2}}{ \partial t_d \partial t_{d-1} \ldots \partial t_3}(\delta^t(f)) \frac{\partial^2}{\partial t_2 \partial t_1}(\delta^t(f)) + \ldots + \frac{\partial}{\partial t_d}(\delta^t(f)) \frac{\partial^{d-1}}{\partial t_{d-1} \partial t_{d-2} \ldots \partial t_1}(\delta^t(f)) 
\end{align}

\begin{align}
\vdots
\nonumber
\end{align}

\begin{align}
& \mathbf{v}^m(f) \coloneq \frac{\partial^m}{\partial t_d \partial t_{d-1} \ldots \partial t_m} (\delta^t(f)) \frac{\partial}{\partial t_{m-1}}(\delta^t(f)) \ldots \frac{\partial}{\partial t_{1}}(\delta^t(f))  \nonumber \\
 & + \frac{\partial^{m-1}}{\partial t_d \partial t_{d-1} \ldots \partial t_{d-m + 1}} (\delta^t(f))  \frac{\partial^2}{\partial t_{d-m} \partial t_m}(\delta^t(f)) \frac{\partial}{\partial t_{m-1}}(\delta^t(f)) \ldots \frac{\partial}{\partial t_{1}}(\delta^t(f))  + \ldots + \nonumber \\ 
 & \frac{\partial}{\partial t_d}(\delta^t(f)) \frac{\partial^m}{\partial t_{d-1} \partial t_{d-2} \ldots \partial t_{d-m-1}}(\delta^t(f)) \frac{\partial}{\partial t_{d-m-2}}(\delta^t(f)) \ldots \frac{\partial}{\partial t_1} (\delta^t(f))   
\end{align}

\begin{align}
\vdots \nonumber
\end{align}

\begin{align}
& \mathbf{v}^{d-1}(f) \coloneq \frac{\partial^2}{\partial t_d \partial t_{d-1}}(\delta^t(f)) \frac{\partial}{\partial t_{d-2}}(\delta^t(f)) \ldots \frac{\partial}{\partial t_1} (\delta^t(f)) +\frac{\partial}{\partial t_d }(\delta^t(f)) \frac{\partial^2}{\partial t_{d-1} \partial t_{d-2}}(\delta^t(f)) \ldots \frac{\partial}{\partial t_1} (\delta^t(f)) \nonumber \\
 & + \ldots +\frac{\partial}{\partial t_d }(\delta^t(f)) \frac{\partial}{\partial t_{d-1}}(\delta^t(f)) \ldots  \frac{\partial}{\partial t_3}(\delta^t(f))\frac{\partial^2}{\partial t_2 \partial t_1} (\delta^t(f)) 
\end{align}
\begin{align}
\mathbf{v}^d(f) \coloneq \frac{\partial}{\partial t_{d}}(\delta^t(f))\frac{\partial}{\partial t_{d-1}}(\delta^t(f))\ldots \frac{\partial}{\partial t_1}(\delta^t(f)) 
\end{align}

Note that the first, $\mathbf{v}^1$, and last, $\mathbf{v}^d$, operators have one term while the operators are comprised of a series - each term in the series consisting of the number of partial derivatives corresponding to the power of the operator (i.e. the term $\mathbf{v}^m$ has $m$ partial derivatives for each term in its series). The identity for $A(f)$ can compactly be written as:
\begin{align}
& A(f) = \sum_{j_d = 0}^{N_d} \sum_{j_{d-1} = 0}^{N_{d-1}}  \ldots \sum_{j_1 = 0}^{N_1} \int_0^1 \int_0^1 \ldots \int_0^1 A'(\delta^t(f))\mathbf{v}^1 + A''(\delta^t(f))\mathbf{v}^2  +\ldots + \nonumber \\
& A^m(\delta^t(f))\mathbf{v}^m +  \ldots + A^{d-1}(\delta^t(f))\mathbf{v}^{d-1} + A^d(\delta^t(f))\mathbf{v}^d dt_1 dt_2 \ldots dt_d
\end{align}

We can now define the approximation, $\tilde{A}_{(N_i)_{i=0}^d}(f)$,  by replacing the sequence $(\mathbf{v}^1,\mathbf{v}^2 \ldots, \mathbf{v}^d)$ with $(\tilde{\mathbf{v}}^1,\tilde{\mathbf{v}}^2, \ldots , \tilde{\mathbf{v}}^d )$ - each element in the latter sequence is comprised with a wavelet coefficient in the direction of the partial derivative for an element in the former sequence.

\begin{align}
\mathbf{\tilde{v}}^1(f) \coloneq Q^{j_d} Q^{j_{d-1}} \ldots Q^{j_1} (f) 
\end{align}

\begin{align}
& \tilde{\mathbf{v}}^2(f) \coloneq Q^{j_d}Q^{j_{d-1}} \ldots Q^{j_2}(f)Q^{j_1}(f) + Q^{j_d}Q^{j_{d-1}} \ldots Q^{j_3}(f)Q^{j_2}Q^{j_1}(f)+ \ldots \nonumber \\
& + Q^{j_d}(f)Q^{j_{d-1}}Q^{j_{d-2}} \ldots Q^{j_1}(f) 
\end{align}

\begin{align}
\vdots
\nonumber
\end{align}

\begin{align}
& \tilde{\mathbf{v}}^m(f) \coloneq Q^{j_d}Q^{j_{d-1}} \ldots Q^{j_m}(f)Q^{j_{m-1}}(f) \ldots Q^{j_1}(f) \nonumber \\
 & + Q^{j_d}Q^{j_{d-1}} \ldots Q^{j_{d-m+1}}(f)Q^{j_{d-m}}Q^{j_m}(f) \ldots Q^{j_{m-1}}(f)  + \ldots + \nonumber \\ 
 & Q^{j_d}(f)Q^{j_{d-1}}Q^{j_{d-2}} \ldots Q^{j_{d-m-1}}(f)Q^{j_{d-m-2}}(f) \ldots Q^{j_1}(f) 
 \label{eq:38}
\end{align}

\begin{align}
\vdots \nonumber
\end{align}

\begin{align}
& \tilde{\mathbf{v}}^{d-1}(f) \coloneq Q^{j_d}Q^{j_{d-1}}(f)Q^{j_{d-2}}(f) \ldots Q^{j_1}(f) +  Q^{j_d}(f)Q^{j_{d-1}}Q^{j_{d-2}}(f) \ldots Q^{j_1}(f) \nonumber \\
 & + \ldots + Q^{j_d}(f) Q^{j_{d-1}}(f)) \ldots  Q^{j_3}(f)Q^{j_2}Q^{j_1}(f) 
\end{align}
\begin{align}
\tilde{\mathbf{v}}^d(f) \coloneq Q^{j_d}(f)Q^{j_{d-1}}(f) \ldots Q^{j_1}(f)  
\end{align}

Again, note the first, $\tilde{\mathbf{v}}^1$, and last, $\tilde{\mathbf{v}}^d$, operators have one term while the operators, $(\mathbf{v}^2, \ldots, \mathbf{v}^{d-1})$, are comprised of a series each with $d-1$ terms. In this case, each term in the series contains the number of wavelet coefficients corresponding to the power of the operator (i.e. the term $\tilde{\mathbf{v}}^m$ has $m$ wavelet coefficients for each term in its series). Now the following linear combination of  $(\tilde{\mathbf{v}}^1,\tilde{\mathbf{v}}^2, \ldots , \tilde{\mathbf{v}}^d )$ is independent of $t$ and is used to construct the approximation, $\tilde{A}_{(N_i)_{i=0}^d}(f)$:

\begin{align}
& \tilde{A}_{(N_i)_{i=0}^d}(f) \coloneq \sum_{j_d = 0}^{N_d} \sum_{j_{d-1} = 0}^{N_{d-1}}  \ldots \sum_{j_1 = 0}^{N_1} A'(P^{j_1,j_2, \ldots, j_d}(f))\tilde{\mathbf{v}} +  A''(P^{j_1,j_2, \ldots, j_d}(f))\tilde{\mathbf{v}}^2 + \ldots + A^m(P^{j_1,j_2, \ldots, j_d}(f))\tilde{\mathbf{v}}^m \nonumber \\
&+ \ldots + A^d(P^{j_1,j_2, \ldots, j_d}(f))\tilde{\mathbf{v}}^d \nonumber \\
& = \sum_{j_d = 0_d}^{N_d} \sum_{j_{d-1} = 0_{d-1}}^{N_{d-1}}  \ldots \sum_{j_1 = 0_1}^{N_1} (\sum_{\beta=1}^d A^{\beta}(P^{j_1,j_2, \ldots, j_d}(f))\tilde{\mathbf{v}}^{\beta})
\end{align}

Computing the residual, $\Delta_{(N_i)_{i=1}^d}(A,f) \coloneq  A(f) - \tilde{A}_{(N_i)_{i=0}^d}(f)$, yields:
\begin{align}
 & \Delta_{(N_i)_{i=1}^d}(A,f) \coloneq \sum_{j_d = 0}^{N_d} \sum_{j_{d-1} = 0}^{N_{d-1}}  \ldots \sum_{j_1 = 0}^{N_1} \int_0^1 \int_0^1 \ldots \int_0^1 (A'(\delta^t(f))\mathbf{v} -  A'(P^{j_1,j_2, \ldots, j_d}(f))\tilde{\mathbf{v}}) + \nonumber \\
 & (A''(\delta^t(f))\mathbf{v}^2 -  A''(P^{j_1,j_2, \ldots, j_d}(f))\tilde{\mathbf{v}}^2) + \ldots + (A^m(\delta^t(f))\mathbf{v}^m -  A^m(P^{j_1,j_2, \ldots, j_d}(f))\tilde{\mathbf{v}}^m) + \ldots +  \nonumber \\
 & (A^{d-1}(\delta^t(f))\mathbf{v}^{d-1} -  A^{d-1}(P^{j_1,j_2, \ldots, j_d}(f))\tilde{\mathbf{v}}^{d-1}) +(A^d(\delta^t(f))\mathbf{v}^d-  A^d(P^{j_1,j_2, \ldots, j_d}(f))\tilde{\mathbf{v}}^d) dt_1 dt_2 \ldots dt_d \nonumber \\
 & = \sum_{j_d = 0}^{N_d} \sum_{j_{d-1} = 0}^{N_{d-1}}  \ldots \sum_{j_1 = 0}^{N_1} \int_0^1 \int_0^1 \ldots \int_0^1  \sum_{m=1}^d (A^{m}(\delta^t(f))\mathbf{v}^{m} -  A^{m}(P^{j_1,j_2, \ldots, j_d}(f))\tilde{\mathbf{v}}^{m}) dt_1 dt_2 \ldots dt_d
 \label{42}
 \end{align}

We characterize the regularity of the residual, $\Delta_{(N_i)_{i=1}^d}(A,f)$, by estimates on the terms in the integrand. Consider an arbitrary term, $(A^m(\delta^t(f))\mathbf{v}^m -  A^m(P^{j_1,j_2, \ldots, j_d}(f))\tilde{\mathbf{v}}^m)$, in the series, $\sum_{m=1}^d (A^{m}(\delta^t(f))\mathbf{v}^{m} -  A^{m}(P^{j_1,j_2, \ldots, j_d}(f))\tilde{\mathbf{v}}^{m}) $, in the integrand. Then

\begin{align}
& \lVert A^m(\delta^t(f))\mathbf{v}^m   \rVert_{L^1([0,1]^d)} \leq \lVert  A^m(\delta^t(f)) \rVert_{L^1([0,1]^d)} \lVert \mathbf{v}^m   \rVert_{L^{\infty}([0,1]^d)} \nonumber \\
\end{align}
from H\"olders inequality. Handling the estimates of $\lVert  A^m(\delta^t(f)) \rVert_{L^1([0,1]^d)} $ and $ \lVert \mathbf{v}^m   \rVert_{L^{\infty}([0,1]^d)}$ separately, we have:
\begin{align}
 & \lVert  A^m(\delta^t(f)) \rVert_{L^1([0,1]^d)} \leq \lVert \sup_{t_1,t_2, \ldots, t_d}   A^m(\delta^t(f))\rVert_{L^1([0,1]^d)} \nonumber \\
 & =  \lVert A^m(Q^{j_1,j_2, \ldots, j_d}(f)) \rVert_{L^1([0,1]^d)} \nonumber \\
 & \leq  \sum_{k_1=1_1}^{2^{j_1}} \sum_{k_2=1_2}^{2^{j_2}}  \ldots \sum_{k_d=1_d}^{2^{j_d}}    \lVert A^m(Q^{j_1,j_2, \ldots, j_d}_{k_1,k_2, \ldots, k_d}(f)) \rVert_{L^1( I^{j_1}_{k_1} \times I^{j_2}_{k_2} \times \ldots \times I^{j_d}_{k_d} )} \nonumber \\
 & =    \sum_{k_1=1_1}^{2^{j_1}} \sum_{k_2=1_2}^{2^{j_2}}  \ldots \sum_{k_d=1_d}^{2^{j_d}}  \int_{  \alpha^{\mathbf{j}}_{\mathbf{k}} ( I^{j_1}_{k_1} \times I^{j_2}_{k_2} \times \ldots \times I^{j_d}_{k_d} )} |A^m (y) | dy \nonumber \\
 & \leq \sum_{k_1=1_1}^{2^{j_1}} \sum_{k_2=1_2}^{2^{j_2}}  \ldots \sum_{k_d=1_d}^{2^{j_d}}  \lVert A^m(Q^{j_1,j_2, \ldots, j_d}_{k_1,k_2, \ldots, k_d}(f)) \rVert_{L^{\infty}( I^{j_1}_{k_1} \times I^{j_2}_{k_2} \times \ldots \times I^{j_d}_{k_d} )} \lVert Q^{j_1,j_2, \ldots, j_d}_{k_1,k_2, \ldots, k_d}(f))  \rVert_{L^{\infty}( I^{j_1}_{k_1} \times I^{j_2}_{k_2} \times \ldots \times I^{j_d}_{k_d} )} \nonumber \\
 &  = C \sum_{k_1=1_1}^{2^{j_1}} \sum_{k_2=1_2}^{2^{j_2}}  \ldots \sum_{k_d=1_d}^{2^{j_d}}   2^{-(j_1 + j_2 + \ldots + j_d) \alpha}
\label{eq:45}
\end{align}
where $   \alpha^{\mathbf{j}}_{\mathbf{k}} $ is the wavelet coefficient at the specified scale and location parameters, the second inequality comes from triangle inequality, the third inequality from the $L^1$ estimate on $A^m$- which is finite since  $A^m(Q^{j_1,j_2, \ldots, j_d}_{k_1,k_2, \ldots, k_d}(f)) \in L^{\infty}$ by the maximum principle - and that wavelet coefficients of H\"older functions can be characterized by exponential decay in the scale parameters \cite{meyer1990ondelettes,stephane1999wavelet}. For $ \lVert \mathbf{v}^m   \rVert_{L^{\infty}([0,1]^d)}$, 

\begin{align}
& \lVert \mathbf{v}^m   \rVert_{L^{\infty}([0,1]^d)} \leq  \lVert \sup_{t_1,t_2, \ldots, t_d} \mathbf{v}^m  \rVert_{L^{\infty}([0,1]^d)} \nonumber \\
& =  (d-1) \lVert   Q^{j_1,j_2, \ldots, j_d} \rVert_{L^{\infty}([0,1]^d)} \nonumber \\
& \leq C \sum_{k_1=1_1}^{2^{j_1}} \sum_{k_2=1_2}^{2^{j_2}}  \ldots \sum_{k_d=1_d}^{2^{j_d}} \lVert Q^{j_1,j_2, \ldots, j_d}_{k_1,k_2, \ldots, k_d}(f))   \rVert_{L^{\infty}( I^{j_1}_{k_1} \times I^{j_2}_{k_2} \times \ldots \times I^{j_d}_{k_d} )} \nonumber \\
& = C \sum_{k_1=1_1}^{2^{j_1}} \sum_{k_2=1_2}^{2^{j_2}}  \ldots \sum_{k_d=1_d}^{2^{j_d}} 2^{-(j_1 + j_2 + \ldots + j_d)\alpha}
\label{eq:44}
\end{align}
where again, the last inequality comes from exponential decay in the wavelet scaling parameter and H\"older functions. Combining the estimates on  $\lVert  A^m(\delta^t(f)) \rVert_{L^1([0,1]^d)} $ and $ \lVert \mathbf{v}^m   \rVert_{L^{\infty}([0,1]^d)}$ yields:

\begin{align}
\lVert A^m(\delta^t(f))\mathbf{v}^m   \rVert_{L^1([0,1]^d)} \leq C 2^{-(j_1 + j_2 + \ldots, j_d)2\alpha}
\label{eq:46}
\end{align}

Now we bound the approximation terms, $ \lVert A^m(P^{j_1,j_2, \ldots, j_d}(f))\tilde{\mathbf{v}}^m) \rVert_{L^{1}([0,1]^d)} $,  for the mth term in the series under the integral for $\Delta_{(N_i)_{i=1}^d}(A,f)$ similarly. By H\"older's inequality, 

\begin{align}
& \lVert A^m(P^{j_1,j_2, \ldots, j_d}(f))\tilde{\mathbf{v}}^m) \rVert_{L^{1}([0,1]^d)} \leq \lVert A^m(P^{j_1,j_2, \ldots, j_d}(f)) \rVert_{L^1([0,1]^d)} \lVert \tilde{\mathbf{v}}^m \rVert_{L^{\infty}([0,1]^d) }
\end{align}
where

\begin{align}
&\lVert A^m(P^{j_1,j_2, \ldots, j_d}(f)) \rVert_{L^1([0,1]^d)} \leq \lVert A^m(Q^{j_1,j_2, \ldots, j_d}(f)) \rVert_{L^1([0,1]^d)} \nonumber \\
& \leq C \sum_{k_1=1_1}^{2^{j_1}} \sum_{k_2=1_2}^{2^{j_2}}  \ldots \sum_{k_d=1_d}^{2^{j_d}}   2^{-(j_1 + j_2 + \ldots + j_d) \alpha}
\label{eq:48}
\end{align}
where the last inequality holds from the same argument in (\ref{eq:45}), and 

\begin{align}
& \lVert \tilde{\mathbf{v}}^m \rVert_{L^{\infty}([0,1]^d) } = \lVert  Q^{j_d}Q^{j_{d-1}} \ldots Q^{j_m}(f) Q^{j_{m-1}}(f) \ldots Q^{j_1}(f) \rVert_{L^{\infty}([0,1]^d) } + \nonumber \\
& \lVert  Q^{j_d}Q^{j_{d-1}} \ldots Q^{j_{d-m+1}}(f) Q^{j_{m-1}}(f) Q^{j_m}(f) \ldots Q^{j_1}(f) \rVert_{L^{\infty}([0,1]^d) } \nonumber \\
& + \ldots + \lVert  Q^{j_d}(f)Q^{j_{d-1}}Q^{j_{d-2}} \ldots Q^{j_{d-m-1}}(f)Q^{j_{d-m-2}}(f) \ldots Q^{j_1}(f) \rVert_{L^{\infty}([0,1]^d) }   \nonumber \\
& \leq C \sum_{k_1=1_1}^{2^{j_1}} \sum_{k_2=1_2}^{2^{j_2}}  \ldots \sum_{k_d=1_d}^{2^{j_d}}   2^{-(j_d + j_{d-1} + \ldots + j_m) \alpha}2^{-j_{m-1}\alpha} \ldots 2^{-j_1\alpha} \nonumber \\ 
& + 2^{-(j_d + j_{d-1} + \ldots + j_{d-m+1}) \alpha}2^{-j_{m-1}\alpha} \ldots 2^{-j_1\alpha} + \ldots + 2^{-j_d \alpha} 2^{-(j_{d-1} + \ldots  + j_{d-m-1})\alpha} 2^{-j_{d-m-2}\alpha} \ldots \ldots 2^{-j_1\alpha} \nonumber \\
&  = C \sum_{k_1=1_1}^{2^{j_1}} \sum_{k_2=1_2}^{2^{j_2}}  \ldots \sum_{k_d=1_d}^{2^{j_d}} 2^{-(j_1 + j_2 +  \ldots + j_d) \alpha}
\end{align}
giving us, 

\begin{align}
 \lVert A^m(P^{j_1,j_2, \ldots, j_d}(f))\tilde{\mathbf{v}}^m) \rVert_{L^{1}([0,1]^d)} \leq C \sum_{k_1=1_1}^{2^{j_1}} \sum_{k_2=1_2}^{2^{j_2}}  \ldots \sum_{k_d=1_d}^{2^{j_d}}   2^{-(j_1 + j_2 + \ldots + j_d) 2 \alpha}
 \label{eq:50}
\end{align}

Combining the inequalities in (\ref{eq:46})  and (\ref{eq:50}) gives:

\begin{align}
& \left| \lVert (A^m(\delta^t(f))\mathbf{v}^m \rVert_{L^{1}([0,1]^d)} - \lVert A^m(P^{j_1,j_2, \ldots, j_d}(f))\tilde{\mathbf{v}}^m) \rVert_{L^{1}([0,1]^d)} \right| \leq  C   2^{-(j_1 + j_2 + \ldots + j_d) 2 \alpha} 
\label{eq:51}
\end{align}

The argument above used to provide estimates for the arbitrary term, $ (A^m(\delta^t(f))\mathbf{v}^m -  A^m(P^{j_1,j_2, \ldots, j_d}(f))\tilde{\mathbf{v}}^m)$, in the integrand holds for the terms in the series, $ \sum_{m=2}^{d-1} (A^{m}(\delta^t(f))\mathbf{v}^{m} -  A^{m}(P^{j_1,j_2, \ldots, j_d}(f))\tilde{\mathbf{v}}^{m})$.  The first and last terms are handled slightly different since $\mathbf{v}^1$ and $\mathbf{v}^d$ contain 1 term in their series instead of $d-1$ terms. Since $\mathbf{v}^1 = \tilde{\mathbf{v}}^1$, the first term becomes: 

\begin{align}
(A'(\delta^t(f)) -  A'(P^{j_1,j_2, \ldots, j_d}(f)))\tilde{\mathbf{v}}^1
\end{align}
Then the following $L^1$ estimate is obtained:

\begin{align}
& \lVert (A'(\delta^t(f)) -  A'(P^{j_1,j_2, \ldots, j_d}(f)))\tilde{\mathbf{v}}^1 \rVert_{L^1([0,1]^d)} \leq \lVert (A'(\delta^t(f)) -  A'(P^{j_1,j_2, \ldots, j_d}(f))) \rVert_{L^1([0,1]^d)} \rVert \tilde{\mathbf{v}}^1 \lVert_{L^\infty([0,1]^d)} \nonumber \\
& = \lVert (A'(\delta^t(f)) -  A'(P^{j_1,j_2, \ldots, j_d}(f))) \rVert_{L^1([0,1]^d)} \lVert Q^{j_1,j_2,\ldots,j_d}(f) \rVert_{L^\infty([0,1]^d)} \nonumber \\
& \leq C  \lVert \delta^t(f)) -  P^{j_1,j_2, \ldots, j_d}(f) \rVert_{L^1([0,1]^d)} \lVert Q^{j_1,j_2,\ldots,j_d}(f) \rVert_{L^\infty([0,1]^d)} \nonumber \\
&  \leq C \lVert Q^{j_1,j_2,\ldots,j_d}(f) \rVert_{L^1([0,1]^d)} \lVert Q^{j_1,j_2,\ldots,j_d}(f) \rVert_{L^\infty([0,1]^d)} \nonumber \\
& =  C \sum_{k_1=1_1}^{2^{j_1}} \sum_{k_2=1_2}^{2^{j_2}}  \ldots \sum_{k_d=1_d}^{2^{j_d}} \lVert Q^{j_1,j_2, \ldots, j_d}_{k_1,k_2, \ldots, k_d}(f))   \rVert_{L^{1}( I^{j_1}_{k_1} \times I^{j_2}_{k_2} \times \ldots \times I^{j_d}_{k_d} )} \lVert Q^{j_1,j_2, \ldots, j_d}_{k_1,k_2, \ldots, k_d}(f))   \rVert_{L^{\infty}( I^{j_1}_{k_1} \times I^{j_2}_{k_2} \times \ldots \times I^{j_d}_{k_d} )} \nonumber  \\
& = C \sum_{k_1=1_1}^{2^{j_1}} \sum_{k_2=1_2}^{2^{j_2}}  \ldots \sum_{k_d=1_d}^{2^{j_d}} \int_{\alpha^{\mathbf{j}}_\mathbf{k}(I^j_k \times I^j_k \times \ldots \times I^j_k)} Q(y) dy \lVert Q^{j_1,j_2, \ldots, j_d}_{k_1,k_2, \ldots, k_d}(f))   \rVert_{L^{\infty}( I^{j_1}_{k_1} \times I^{j_2}_{k_2} \times \ldots \times I^{j_d}_{k_d} )} \nonumber \\
& \leq C \sum_{k_1=1_1}^{2^{j_1}} \sum_{k_2=1_2}^{2^{j_2}}  \ldots \sum_{k_d=1_d}^{2^{j_d}} \lVert Q^{j_1,j_2, \ldots, j_d}_{k_1,k_2, \ldots, k_d}(f))   \rVert_{L^{\infty}( I^{j_1}_{k_1} \times I^{j_2}_{k_2} \times \ldots \times I^{j_d}_{k_d} )} \lVert Q^{j_1,j_2, \ldots, j_d}_{k_1,k_2, \ldots, k_d}(f))   \rVert_{L^{\infty}( I^{j_1}_{k_1} \times I^{j_2}_{k_2} \times \ldots \times I^{j_d}_{k_d} )} \nonumber \\
& \leq C  \sum_{k_1=1_1}^{2^{j_1}} \sum_{k_2=1_2}^{2^{j_2}}  \ldots \sum_{k_d=1_d}^{2^{j_d}}  2^{-(j_1 + j_2 + \ldots + j_d)2\alpha}
\end{align}
by application of H\"older's inequality, the property of $A \in C^d$, and the estimates on the wavelet coefficients of H\"older functions previously discussed. So by the previous result, the first term in the series under the integrand of $\Delta_{(N_i)_{i=1}^d}(A,f)$ has the estimate:

\begin{align}
& \left| \lVert (A'(\delta^t(f)) - A'(P^{j_1,j_2, \ldots, j_d}(f)))\tilde{\mathbf{v}}) \rVert_{L^{1}([0,1]^d)} \right| \leq  C   2^{-(j_1 + j_2 + \ldots + j_d) 2 \alpha} 
\label{eq:54}
\end{align}

Now we handle the final term in the series, $(A^d(\delta^t(f)) \mathbf{v}^d -  A^d(P^{j_1,j_2, \ldots, j_d}(f))\tilde{\mathbf{v}}^d$, where $\tilde{\mathbf{v}}^d$ only has 1 term in its series. Since $\mathbf{v} \neq \tilde{\mathbf{v}}^d$, we handle $(A^d(\delta^t(f)) \mathbf{v}^d$ and $A^d(P^{j_1,j_2, \ldots, j_d}(f))\tilde{\mathbf{v}}^d$ separately as we did for the $m$th term; the obtained estimates are the same but the technology used is negligibly different. The $\lVert A^d(\delta^t(f)) \rVert_{L^1([0,1])}$ estimate is the same as (\ref{eq:44}), but $ \lVert \mathbf{v}^d \rVert_L^{\infty}([0,1])$ is slightly different:

\begin{align}
& \lVert \mathbf{v}^d   \rVert_{L^{\infty}([0,1]^d)} \leq  \lVert \sup_{t_1,t_2, \ldots, t_d} \mathbf{v}^d  \rVert_{L^{\infty}([0,1]^d)} \nonumber \\
& =  \lVert   Q^{j_1,j_2, \ldots, j_d} \rVert_{L^{\infty}([0,1]^d)} \nonumber \\
& \leq C \sum_{k_1=1_1}^{2^{j_1}} \sum_{k_2=1_2}^{2^{j_2}}  \ldots \sum_{k_d=1_d}^{2^{j_d}} \lVert Q^{j_1,j_2, \ldots, j_d}_{k_1,k_2, \ldots, k_d}(f))   \rVert_{L^{\infty}( I^{j_1}_{k_1} \times I^{j_2}_{k_2} \times \ldots \times I^{j_d}_{k_d} )} \nonumber \\
& = C \sum_{k_1=1_1}^{2^{j_1}} \sum_{k_2=1_2}^{2^{j_2}}  \ldots \sum_{k_d=1_d}^{2^{j_d}} 2^{-(j_1 + j_2 + \ldots + j_d)\alpha}
\end{align}

Thus, we again have:

\begin{align}
\lVert A^d(\delta^t(f))\mathbf{v}^d  \rVert_{L^1([0,1]^d)} \leq C 2^{-(j_1 + j_2 + \ldots, j_d)2\alpha}
\label{eq:55}
\end{align}
as in (\ref{eq:46}). Similar reasoning applies to the $A^d(P^{j_1,j_2, \ldots, j_d}(f))\tilde{\mathbf{v}}^d$ term. The $ \lVert A^d(P^{j_1,j_2, \ldots, j_d}(f)) \rVert_{L^1([0,1]^d)}$ estimate is the same as (\ref{eq:48}), but for $ \lVert \tilde{\mathbf{v}}^d \rVert_{L^{\infty}([0,1]^d)}$ we have:

\begin{align}
& \lVert \tilde{\mathbf{v}}^d \rVert_{L^{\infty}([0,1]^d)} = \lVert Q^{j_d}(f)Q^{j_{d-1}}(f) \ldots Q^{j_1}(f) \rVert_{L^{\infty}([0,1]^d)} \nonumber \\
& \leq \sum_{k_d=1_d}^{2^{j_d}} Q^{j_d}_{k_d}(f) \sum_{k_{d-1}=1_{d-1}}^{2^{j_{d-1}}} Q^{j_{d-1}}_{k_{d-1}}(f) \ldots \sum_{k_1=1_1}^{2^{j_1}} Q^{j_1}_{k_1}(f) \nonumber \\
& \leq C \sum_{k_d=1_d}^{2^{j_d}} 2^{-j_d\alpha} \sum_{k_{d-1}=1_{d-1}}^{2^{j_{d-1}}} 2^{-j_{d-1}\alpha} \ldots \sum_{k_1=1_1}^{2^{j_1}} 2^{-j_1\alpha} \nonumber \\
& = C  2^{-(j_1 + j_2 + \ldots + j_d)\alpha}
\label{eq:56}
\end{align}
Combining (\ref{eq:55}) and (\ref{eq:56}), the last term of the series in the integrand of $\Delta_{(N_i)_{i=1}^d}(A,f)$ has the estimate:

\begin{align}
& \left| \lVert (A^d(\delta^t(f))\mathbf{v}^d \rVert_{L^{1}([0,1]^d)} - \lVert A^d(P^{j_1,j_2, \ldots, j_d}(f))\tilde{\mathbf{v}}^d) \rVert_{L^{1}([0,1]^d)} \right| \leq  C   2^{-(j_1 + j_2 + \ldots + j_d) 2 \alpha} 
\label{eq:58}
\end{align}

So by  (\ref{eq:51}),  (\ref{eq:54}), and  (\ref{eq:58}) we can express the residual, $\Delta_{(N_i)_{i=1}^d}(A,f)$, defined in (\ref{42}), as a sum of piecewise constant functions:

\begin{align}
 & \Delta_{(N_i)_{i=1}^d}(A,f) = \sum_{j_d = 0}^{N_d} \sum_{j_{d-1} = 0}^{N_{d-1}}  \ldots \sum_{j_1 = 0}^{N_1}  (\sum_{k_1=1_1}^{2^{j_1}} \sum_{k_2=1_2}^{2^{j_2}}  \ldots \sum_{k_d=1_d}^{2^{j_d}}  \eta^{j_1,j_2, \ldots, j_d}_{k_1,k_2, \ldots, k_d} \chi_{I^{j_1}_{k_1} \times I^{j_2}_{k_2}  \times \ldots \times I^{j_d}_{k_d}} (x_1,x_2, \ldots, x_d) ), \nonumber \\
 &|\beta^{j_1,j_2, \ldots, j_d}_{k_1,k_2, \ldots, k_d} | \leq C 2^{-(j_1 + j_2 + \ldots j_d)2\alpha}
 \label{eq:59}
\end{align}

Now we characterize the regularity of $\Delta_{(N_i)_{i=1}^d}(A,f)$ with the following lemma:

\begin{lemma}

The residual, $\Delta_{(N_i)_{i=1}^d}(A,f) = A(f) - \tilde{A}_{(N_i)_{i=0}^d}(f)$, can be writen as:
\begin{align}
& \Delta_{(N_i)_{i=1}^d}(A,f) = \sum_{j_1, \ldots, j_d = 0}^{N_1, \ldots, N_d}  (\sum_{k_1, \ldots, k_d=1}^{2^{j_1}, \ldots, 2^{j_d}} \eta^{(j_i)_{i=1}^d}_{(k_i)_{i=1}^d} \chi_{I^{j_1}_{k_1} \times I^{j_2}_{k_2}  \times \ldots \times I^{j_d}_{k_d}} (x_1,x_2, \ldots, x_d) ), \nonumber \\
 &| \eta^{(j_i)_{i=1}^d}_{(k_i)_{i=1}^d} | \leq C 2^{-(j_1 + j_2 + \ldots j_d)2\alpha}
 \label{eq:60}
 \end{align}
and is an element of $\Lambda_{2\alpha}([0,1]^d)$
 \label{lemma:4}
\end{lemma}

\begin{proof}

To show $\Delta_{(N_i)_{i=1}^d}(A,f) \in \Lambda_{2\alpha}([0,1]^d)$, we need to show 

\begin{align}
|\Delta_{(N_i)_{i=1}^d}(A,f)(\mathbf{x}) - \Delta_{(N_i)_{i=1}^d}(A,f)(\mathbf{\tilde{x}})| \leq C d_d(x'_1,x_1)^{2\alpha}d_d(x'_2, x_2)^{2\alpha} \ldots d_d(x'_d, x_d)^{2\alpha}
\label{eq:61}
\end{align}
where $\mathbf{x} = (x_1, x_2, \ldots, x_d) \in [0,1]^d, \mathbf{\tilde{x}} = (\tilde{x}_1, \tilde{x}_2, \ldots, \tilde{x}_d) \in [0,1]^d$.  To show this, we consider two arbitrary points $\mathbf{x}$ , $\mathbf{\tilde{x}} \in [0,1]^d$ and show that the relation in (\ref{eq:61}) is always satisfied. Suppose we select two arbitrary points $\mathbf{x}, \mathbf{\tilde{x}} \in [0,1]^d$. There are two cases two consider; either $\mathbf{x}, \mathbf{\tilde{x}}$ are in the same dyadic box or not. Let $B^{j_1,j_2, \ldots, j_d}_{k_1,k_2, \ldots, k_d} \coloneq I^{j_1}_{k_1} \times I^{j_2}_{k_2} \times  \ldots \times I^{j_d}_{k_d} = B^{\mathbf{j}}_{\mathbf{k}}$ be the cartesian product of $d$ dyadic intervals for each direction.  Let $\hat{\mathbf{j}} = (\hat{j_1},\hat{j_2}, \ldots, \hat{j_d})$ be the sequence of scales at which $\mathbf{x}, \mathbf{\tilde{x}}$ are no longer contained in the same dyadic box regardless of the location parameters $ \mathbf{k} = (k_1,k_2, \ldots, k_d)$ associated with each scale in $\hat{\mathbf{j}}$ . Define $\mathcal{U}$ to be the set containing combinations of scaling and location parameters such that $\mathbf{x}, \mathbf{\tilde{x}} $ are always contained in the associated dyadic box:

\begin{align}
\mathcal{U} \coloneq \{ ((j_1,j_2, \ldots, j_d), (k_1, k_2, \ldots, k_d)) :  |B^{\mathbf{j}}_{\mathbf{k}} |  >  |B^{\hat{\mathbf{j}}}_{\mathbf{k}} |  \}
\end{align}
Similarly, let $\mathcal{L}$ be the set containing combinations of scaling and location parameters such that $\mathbf{x}, \tilde{\mathbf{x}}$ are never contained in the same box. 

\begin{align}
L \coloneq \{ ((j_1,j_2, \ldots, j_d), (k_1, k_2, \ldots, k_d)) :  |B^{\mathbf{j}}_{\mathbf{k}} |  \leq |B^{\hat{\mathbf{j}}}_{\mathbf{k}} |  \}
\end{align}
Then summing over the scaling and location parameters in $\mathcal{U}$ for $\Delta_{(N_i)_{i=1}^d}(A,f)$, one has:

\begin{align}
& \sum_{ (\mathbf{j}, \mathbf{l})  \in \mathcal{U}}  \gamma^{\mathbf{j}}_{\mathbf{k}} \chi_{B^{\mathbf{j}}_{\mathbf{k}}}(\mathbf{x}) -  \gamma^{\mathbf{{j}}}_{\mathbf{k}} \chi_{B^{\mathbf{j}}_{\mathbf{k}}}(\mathbf{\tilde{x}}) = 0 \nonumber \\
& \leq C d_d(x'_1,x_1)^{2\alpha}d_d(x'_2, x_2)^{2\alpha} \ldots d_d(x'_d, x_d)^{2\alpha}
\label{eq:64}
\end{align}

since $\Delta_{(N_i)_{i=1}^d}(A,f)$ is piecewise constant on dyadic boxes (see \ref{eq:59}).  Now consider the case where $x,x'$ are not contained in same dyadic box. Again, summing over the scaling and location parameters in $\mathcal{L}$ for $\Delta_{(N_i)_{i=1}^d}(A,f)$ one has: 

\begin{align}
 & \sum_{ (\mathbf{j}, \mathbf{l})  \in \mathcal{L}}  \gamma^{\mathbf{j}}_{\mathbf{k}} \chi_{B^{\mathbf{j}}_{\mathbf{k}}}(\mathbf{x}) -  \gamma^{\mathbf{j}}_{\mathbf{k}} \chi_{B^{\mathbf{j}}_{\mathbf{k}}}(\mathbf{\tilde{x}})  \leq C \sum_{\mathbf{j}^{(i)}= \mathbf{j}^{(0)}}^{\mathbf{j}^{(N)}} 2^{-(\sum_{l}^d \mathbf{j}_l^{(i)})2\alpha} \nonumber \\
 & = \sum_{\mathbf{j}^{(i)} = \mathbf{j}^{(0)}}^{\mathbf{j}^{(N)}} 2^{-(j^{(i)}_1 + j^{(i)}_2 + \ldots + j^{(i)}_d)2 \alpha} \nonumber \\
 & \leq C 2^{-(j^{(N)}_1 + j^{(N)}_2 + \ldots + j^{(N)}_d)2 \alpha} \nonumber \\
 & \leq C d_d(x'_1,x_1)^{2\alpha}d_d(x'_2,x_2)^{2\alpha} \ldots d_d(x'_d,x_d)^{2\alpha}
 \label{eq:65}
\end{align}

where we order the scales in $\mathcal{L}$ such that $\mathbf{j}^{(0)} \leq \mathbf{j}^{(1)} \leq  \ldots \leq \mathbf{j}^{(N)}$ indicates $2^{-(\sum_{l}^d \mathbf{j}_l^{(0)})} \leq 2^{-(\sum_{l}^d \mathbf{j}_l^{(1)})} \leq \ldots \leq 2^{-(\sum_{l}^d \mathbf{j}_l^{(N)})}$. The first inequality holds from the estimates on $\Delta_{(N_i)_{i=1}^d}(A,f)$ obtained in (\ref{eq:59}), the second inequality from collapsing the geometric series, the third from the assumption on the scaling parameters contained in $\mathcal{L}$. Thus, combining the previous two estimates in (\ref{eq:64}) and (\ref{eq:65}) one has:

\begin{align}
 & \Delta_{(N_i)_{i=1}^d}(A,f) = \sum_{ (\mathbf{j}, \mathbf{l})  \in \mathcal{L}}  \gamma^{\mathbf{j}}_{\mathbf{k}} \chi_{B^{\mathbf{j}}_{\mathbf{k}}}(\mathbf{x}) -  \gamma^{\mathbf{j}}_{\mathbf{k}} \chi_{B^{\mathbf{j}}_{\mathbf{k}}}(\mathbf{\tilde{x}}) + \sum_{ (\mathbf{j}, \mathbf{l})  \in \mathcal{U}}  \gamma^{\mathbf{j}}_{\mathbf{k}} \chi_{B^{\mathbf{j}}_{\mathbf{k}}}(\mathbf{x}) -  \gamma^{\mathbf{{j}}}_{\mathbf{k}} \chi_{B^{\mathbf{j}}_{\mathbf{k}}}(\mathbf{\tilde{x}}) \nonumber \\
 & \leq C d_d(x'_1,x_1)^{2\alpha}d_d(x'_2,x_2)^{2\alpha} \ldots d_d(x'_d,x_d)^{2\alpha}
\end{align}

Since the choice of $\mathbf{x}, \mathbf{\tilde{x}}$ were arbitrarily chosen, $\Delta_{(N_i)_{i=1}^d}(A,f) \in \Lambda_{2\alpha}([0,1]^d)$ 

\end{proof}

Thus, from lemma \ref{lemma:4}, we can compute the approximation $\tilde{A}_{(N_i)_{i=0}^d}(f) =  \sum_{\beta=1}^d A^{\beta}(P^{j_1,j_2, \ldots, j_d}(f)) \tilde{\mathbf{v}}^{\beta}$ to $A(f)$ such that $\Delta_{(N_i)_{i=1}^d}(A,f) = A(f) - \tilde{A}_{(N_i)_{i=0}^d}(f) \in \Lambda_{2\alpha}([0,1]^d)$.

\end{proof}

\section{Paraproduct Caricatures}

Included in this section are caricatures of d-tensor paraproducts stated in the main theorem.

\begin{corollary}
Suppose $A \in C^{3}(\mathbb{R}), f \in \Lambda_{2\alpha}([0,1]^3)$, then we obtain the following discrete $3$-tensor paraproduct decomposition

\begin{align}
& \tilde{A}_{(N_i)_{i=0}^d}(f) =  \sum_{j_1,j_2,j_3=0}^{N_1,N_2,N_3} A'(P^{j_1,j_2,j_3}(f))Q^{j_1}(f)Q^{j_2}(f)Q^{j_3}(f) +  \nonumber \\
& + A''(P^{j_1,j_2,j_3}(f))[Q^{j_2}[Q^{j_1}(f)Q^{j_3}(f)] + Q^{j_3}[Q^{j_2}(f)Q^{j_1}(f)]] \nonumber \\
&+  A'''(P^{j_1,j_2,j_3}(f))Q^{j_3}Q^{j_2}Q^{j_1}(f)
\end{align}

\end{corollary}

\begin{proof}
The proof follows from the same techniques used to prove the general $d$-tensor case by setting $d=3$
\end{proof}

\begin{corollary}
The approximation, $\tilde{A}_{(N_i)_{i=0}^d}(f)$, is real analytic on the image of $f$ and has the following taylor expansion

\begin{align}
\tilde{A}_{(N_i)_{i=0}^d}(f) = \sum_{\beta=1}^d A^{\beta}(P^{j_1,j_2, \ldots, j_d}(f)) \tilde{\mathbf{v}}^{\beta}
\end{align}

where $A^{\beta}(P^{j_1,j_2, \ldots, j_d}(f))$ are the derivatives of A of order $\beta$ evaluated at the average of $f$ and $\tilde{\mathbf{v}}^{\beta}$ are wavelet operators acting in $\beta$ directions. 
\end{corollary}

\begin{proof}
As any real analytic function can be represented by its taylor series, it suffices to show $A$ is real analytic on $f$. Since $  \tilde{A}_{(N_i)_{i=0}^d}(f) = \sum_{\beta=1}^d A^{\beta}(P^{j_1,j_2, \ldots, j_d}(f))\tilde{\mathbf{v}}^{\beta}$ by Theorem \ref{thm:4.1} the power series representation $ \sum_{\beta=1}^d A^{\beta}(P^{j_1,j_2, \ldots, j_d}(f))\tilde{\mathbf{v}}^{\beta}$ is immediate, as $\beta$ represents the derivatives of $A$ and the powers of $\tilde{\mathbf{v}} $
\end{proof}

\section{Cone singularity in complex plane (d=3)}

For the 3D example, we consider a time-varying singularity in the complex plane:

\begin{align}
f(z) = 
\begin{cases}
 (1 - |z|)^{\alpha} & \text{if } |z| < 1 \\
 (1 - \frac{1}{|z|})^{\alpha} & \text{if }  \frac{1}{|z|} < 1
\end{cases}
\end{align}

where $|z| = \sqrt{x^2 + iy^2 + t^2}$ is the usual modulus in the complex plane with a time-dependent parameter and $x,iy \in [-1,1], t \in [0,1]$. $N=256$ equispaced points are sampled in the x and y directions, 128 points are sampled in the time direction. Here, $A(f)$ is the same time-dependent graphic equalization function described in \cite{fasina2025quasilinearization}, where the amplitude of $f$ is differentially modulated at different frequency bands given by the scales, $j_x$ and $j_{iy}$, of the real and imaginary components:

\begin{align}
A_3(f) = sin((\frac{j_x + j_{iy}}{2}<f(z),\psi^{j_x,j_{iy}'}(x,iy)>\psi^{j_x,j_{iy}'}(x,iy))
\end{align}

\begin{figure}[H]
    \centering
    \includegraphics[width=0.99\linewidth]{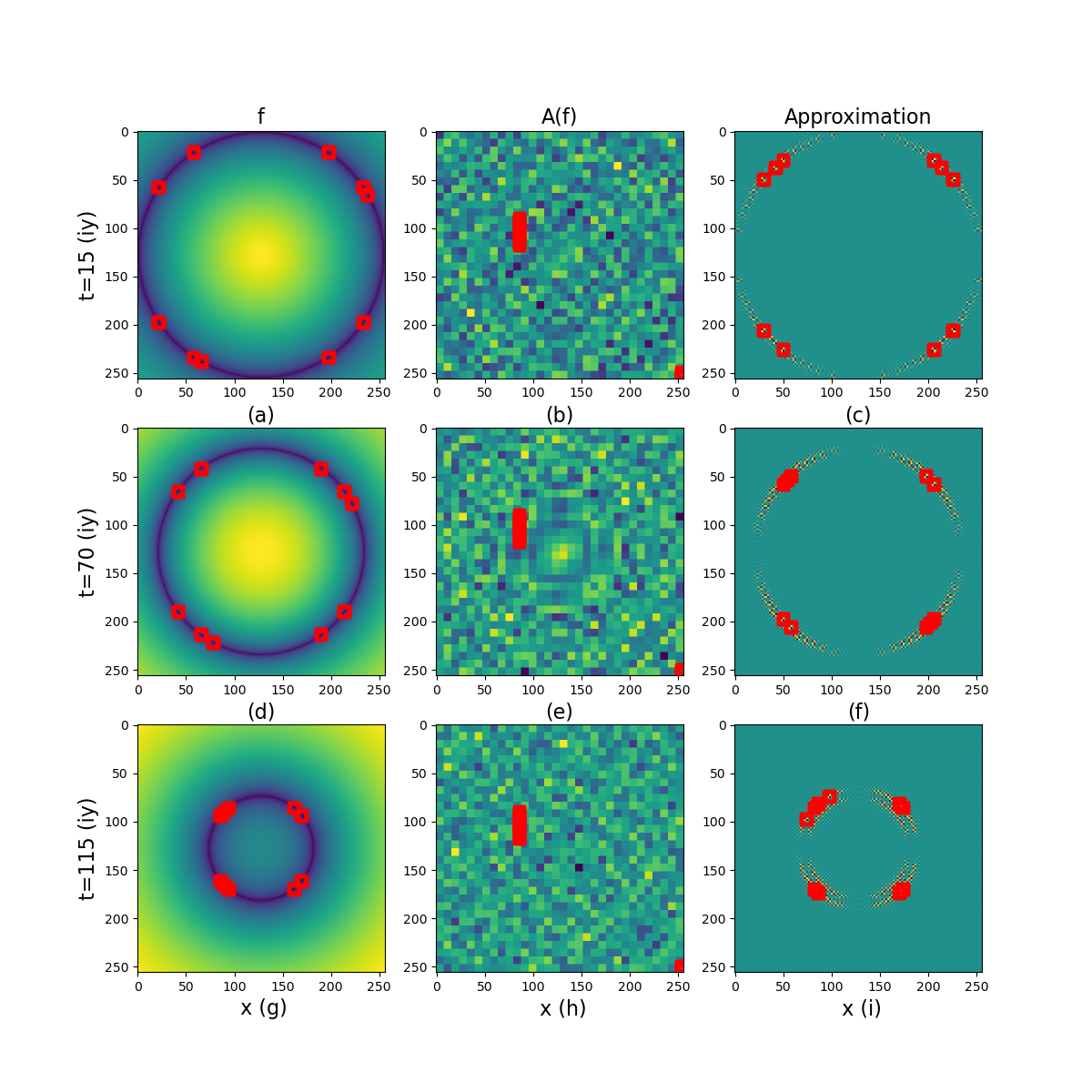}
    \caption{ Cone singularity in complex plane: (a - c) $f,  A(f), \tilde{A}_{(N_i)_{i=0}^d}(f)$ at t=15 (d- f) $f,  A(f), \tilde{A}_{(N_i)_{i=0}^d}(f)$ at t=70, (g-i) $f,  A(f), \tilde{A}_{(N_i)_{i=0}^d}(f)$ at t=115.}
    \label{fig:timevar_complex}
\end{figure}

The left column of Figure \ref{fig:timevar_complex} shows the original singularity progressing towards the center, the middle column the nonlinear distortion, the right column the approximation. The rows represent discrete time points, $t= \{ 15,70, 115 \}$. Qualitatively, one can see the approximation recovers the ring singularity which is conspicuously distorted by $A(f)$.\\

We process all the figures in Figure \ref{fig:timevar_complex} with the tensor haar basis. The red squares are the support of the tensor haar functions with the largest coefficients. This exemplifies the ability of the decomposition to recover the singularity in the presence of smooth non-linear distortions - the singularity is detected as expected for the original function, $f$, goes undetected for the non-linear distortion in the middle column, $A(f)$, and is recovered again by the approximation, $\tilde{A}_{(N_i)_{i=0}^d}(f)$ in the right column.

\bibliographystyle{unsrt}
\bibliography{main}

@inproceedings{bony1981calcul,
  title={Calcul symbolique et propagation des singularit{\'e}s pour les {\'e}quations aux d{\'e}riv{\'e}es partielles non lin{\'e}aires},
  author={Bony, Jean-Michel},
  booktitle={Annales scientifiques de l'{\'E}cole normale sup{\'e}rieure},
  volume={14},
  number={2},
  pages={209--246},
  year={1981}
}

@misc{stephane1999wavelet,
  title={A wavelet tour of signal processing},
  author={ Mallat, Stephane},
  year={1999},
  publisher={Elsevier}
}

@article{meyer1990ondelettes,
  title={Ondelettes et op{\'e}rateurs},
  author={Meyer, Yves},
  journal={I: Ondelettes},
  year={1990},
  publisher={Hermann}
}

@article{fasina2025quasilinearization,
  title={Quasilinearization with Regularizing Tensor Paraproducts},
  author={Fasina, Oluwadamilola},
  journal={arXiv preprint arXiv:2503.12629},
  year={2025}
}

\end{document}